\documentclass[10pt]{article}
\usepackage{amssymb,amsmath,amsopn,amsthm}
\usepackage[mathscr]{euscript}
\usepackage{tikz-cd}
\newtheorem{lemma}{Lemma}
\theoremstyle{definition}
\newtheorem{defn}{Definition}
\newtheorem{Theorem}{Theorem}
\newtheorem{Prop}{Proposition}

\theoremstyle{remark}
\newtheorem{rmk}{Remark}

\title{The Cumulant Bijection and Differential Forms}
\date{}
\author{Nissim Ranade and Dennis Sullivan}

\begin{document}
\maketitle
\section*{Description}

Given a graded commutative algebra $A$ there is a canonical (tautologous) map $\tau:SA\rightarrow A$ where $SA$ is the graded commutative algebra generated by $A$. The map $\tau$ is given by the formula \[\tau(x_1\wedge x_2\wedge\ldots x_n) = x_1x_2\ldots x_n \]

$SA = A \oplus A\wedge A \oplus\ldots$ also has a coproduct structure $\Delta:SA\rightarrow SA\wedge SA$ given by the formula \[\Delta(x_1\wedge x_2\wedge\ldots x_n) = \sum \epsilon\cdot x_{\pi_1}\wedge x_{\pi_2} \] where $\pi_1\sqcup\pi_2$ is a partition of the indices into two arbitrary subsets and $\epsilon$ is the sign that arises when the first subset is moved all the way to the left. One knows that any map like $\tau:SA\rightarrow A$ lifts to a canonical coalgebra mapping $\tilde{\tau}:SA\rightarrow SA$ so that $\pi\circ\tilde{\tau} = \tau$ where $\pi:SA\rightarrow A$ is projection onto $A$, the linear summand of $SA$. 

We verify that $\tilde{\tau}$ is given by the formula 
\begin{equation}
\label{formula}
 \tilde{\tau}(v) = \tau(v) +  \tau\wedge\tau \circ \Delta(v) + \tau^{\wedge 3}\circ \Delta^2(v) + \ldots 
\end{equation}

Note that $\tilde{\tau}(v) = v + \text{lower order terms}$. This implies that it is a bijection. We call this coalgebra isomorphism $\tilde{\tau}$ the \textit{cumulant bijection}.  

The relation of $\tilde{\tau}$ to cumulants appears when it is used to change coordinates in $SA$, in particular when it is used to conjugate extensions to $SA$ of a linear mapping of $A$ to either a coderivation of $SA$ or to a coalgebra mapping of $SA$. 

General coderivations $D:SA\rightarrow SA$ have canonical decompositions by orders $D=D_1+D_2+\ldots$ where $D_k$ is determined by ``Taylor coefficients" $D^k$; 

\[D^1:A\rightarrow A\] \[D^2:A\wedge A\rightarrow A\] \[D^3:A\wedge A\wedge A\rightarrow A\] and so on. 

The explicit formulae relating $D_k$ to $D^k$ are \[ D_k(x_1\wedge\ldots\wedge x_n) = \sum \epsilon\cdot D^k(x_I)\wedge\ldots\hat{x_I}\ldots\]
where the sum is taken over all subsets $I$ of size $k$ of the indexing set. 

A coalgebra mapping $G:SA \rightarrow SA$ is also determined by independent maps (also called the Taylor coefficients) \[G^1:A\rightarrow A\]\[G^2:A\wedge A\rightarrow A\] \[G^3:A\wedge A\wedge A\rightarrow A\] and so on by the formula \[G(x_1\wedge x_2\wedge\ldots) = \sum_{\text{partitions}}\epsilon\cdot G^{j_1}(x_{\pi_{j_1}})\wedge G^{j_2}(x_{\pi_{j_2}})\ldots\]

The canonical extension of a map $A\rightarrow A$ to either a coderivation or a coalgebra mapping, only has a non-zero leading order coefficient. After conjugating the canonical extensions by $\tilde{\tau}$ they have in general Taylor coefficients of all orders. 

In the case of a coalgebra extension $\tilde{f}$ of $f:A\rightarrow A$ the conjugated $g=\tilde{\tau}^{-1}\tilde{f}\tilde{\tau}$ has coefficients which measure the deviation of $f$ from being an algebra homomorphism of $A$. 

The first few Taylor coefficients of $g$ are \[g^1(x)=f(x)\]\[g^2(x,y) = f(xy)-f(x)f(y)\]\[g^3(x,y,z)=f(xyz)-f(xy)f(z)-f(yz)f(x)-f(zx)f(y)+2f(x)f(y)f(z)\]

In the case of a coderivation extension $\hat{f}$ of $f$, the first few coefficients of the conjugate $h=\tilde{\tau}^{-1}\hat{f}\tilde{\tau}$ are \[h^1(x)=f(x)\]\[h^2(x,y) = f(xy)-f(x)y-xf(y)\]\[h^3(x,y,z)=f(xyz)-f(xy)z+xyf(z)-f(yz)x+yxf(x)-f(zx)y+zxf(y)\]

In the first case all the higher Taylor coefficients vanish iff $f(xy)-f(x)f(y)=0$, that is $f$ is an algebra map.

In the second case all higher Taylor coefficients vanish iff $f(xy)-f(x)y-xf(y)=0$, that is $f$ is a derivation. 

\section{First Application}
Let us now consider the case where $A$ is a chain complex with a differential $\partial$ which is not necessarily a derivation of the product structure. Suppose $i:(C,\partial_C)\rightarrow(A,\partial)$ is a sub-complex of $(A,\partial)$ such that $A$ deformation retracts to $C$. That is, there is $I:(A,\partial)\rightarrow(C,\partial_C)$ so that $I\circ i=Id_C$ and there is a degree $+1$ mapping $s:A\rightarrow A$ so that $\partial s+s\partial=i\circ I-Id_A$. Now let $\tilde{I}$ and $\tilde{i}$ be the canonical lifts of $I$ and $i$ to coalgebra maps between $SA$ and $SC$.

\[
\begin{tikzcd}
SA \arrow[bend left]{d}{\tilde{I}}\arrow{r}[swap]{\pi} & A \arrow[bend left]{d}{I}\\
SC \arrow{u}{\tilde{i}} \arrow{r}[swap]{\pi} & C \arrow{u}{i} \\
\end{tikzcd}
\]

Let $d$ be the extension of $\partial$ to $SA$ and $\tilde{d} = \tilde{\tau}^{-1}d\tilde{\tau}$, Then by our previous discussion the cumulant bijection $\tilde{\tau}$ gives an isomorphism between $(SA,d)$ and $(SA,\tilde{d})$.
 
\begin{Theorem}  
\label{app}
Let $C$ be as above. Suppose there is a coderivation $\partial_{\infty}$ on $SC$ and a differential graded coalgebra map $\iota$ from $(SC,\partial_{\infty})$ to $(SA, \tilde{d})$ which extends $i$. Then there is an induced isomorphism, ``the induced cumulant bijection" $\tilde{\tau}_C$ between $(SC, \partial_{\infty})$ and $(SC,d_C)$ where $d_C$ is the canonical coderivation extension of $\partial_C$. $\tilde{\tau}_C$ is uniquely characterized by the commutativity of the following diagram.

\[
\begin{tikzcd}
(SA,\tilde{d}) \arrow{r}{\tilde{\tau}} & (SA,d) \arrow{d}{\tilde{I}}\\
(SC,\partial_{\infty}) \arrow{u}{\iota} \arrow{r}[swap]{\tilde{\tau}_C} & (SC,d_C) \\
\end{tikzcd}
\]
\end{Theorem}

\begin{proof}

As $\iota$ is an injection and $\tilde{I}$ is a surjection, $\tilde{\tau}_C$ must be defined to be equal to $ \tilde{I}\circ\tilde{\tau}\circ\iota$. The fact that $\tilde{\tau}_C$ is an isomorphism follows as in Lemma \ref{iso}, from the fact that since $\iota$ agrees with $i$ on the linear terms and $I\circ i$ is identity, $\tilde{\tau}_C$ induces the identity map on linear terms. Then since $\tilde{\tau}_C$ is a coalgebra map the inductive step of Lemma \ref{iso} holds. For more details see section \ref{Proof}.

\end{proof}

\begin{rmk}
Note that $SC$ receives the induced cumulant bijection from the map $\iota$ and not from a commutative algebra structure on $C$. 
\end{rmk}

\section{The missing details, the cumulant bijection in the associative context and more applications}

\begin{defn}
A \textit{graded coassociative conilpotent coalgebra} is a graded vector space $C$ together with a degree zero coassociative coproduct $\Delta:C\rightarrow C\otimes C$ with the property that for all $v$ in $C$ there exists an integer $n$ such that $\Delta^n(v)=0$. If the image of the coproduct is in the graded symmetric tensors then $C$ is called the \textit{graded symmetric coassociative conilpotent coalgebra}

\end{defn}

\begin{defn}
The \textit{free coassociative conilpotent coalgebra without co-unit} generated by a graded vector space $V$ is a nilpotent coalgebra $T^cV$ with a projection map onto $V$ with the following universal property: Given a linear map from a graded nilpotent coassociative coalgebra to $V$ there exists a unique coalgebra map $\tilde{f}$ from $C$ to $T^cV$ such that the following diagram commutes.

\begin{equation}
\begin{tikzcd}
 {} & C \arrow{d}{f} \arrow[dashed,swap]{dl}{\exists \tilde{f}} \\
  T^cV \arrow{r}{\pi} & V
\end{tikzcd}
\label{univ}
\end{equation}
 
\end{defn}

On the space $TV = V\oplus V^{\otimes2}\oplus V^{\otimes3}\ldots$ consider the coproduct $\Delta$ given by the following formula. \[ \Delta(x_1 \otimes x_2 \otimes\ldots\otimes x_n) = \sum_{i=1}^{n-1} x_1 \otimes \ldots x_i\bigotimes\ldots\otimes x_n \] $TV$ with the coproduct $\Delta$ is the universal free conilpotent coalgebra associated to $V$. The formula for the lift is given later in Lemma \ref{lift}. 

\begin{defn}
The \textit{free graded coassociative cocommutative nilpotent coalgebra} generated by $V$ is a sub-coalgebra $S^cV$ of $T^cV$ such that if $C$ taken as above is also graded cocommutative then the image of $\tilde{f}$ lies in $S^cV$, that is the following diagram commutes.

\begin{equation}
\begin{tikzcd}
 {} & C \arrow{d}{f} \arrow[dashed,swap]{dl}{\exists \tilde{f}} \\
  S^cV \arrow{r}{\pi} & V
\end{tikzcd}
\end{equation}
 
\end{defn}
 
\[S^cV = \bigoplus_{n=1}^\infty S^nV\] where $S^nV$ is the subspace of $V^{\otimes n}$ generated by \[\sum_{\sigma \in S_n} \pm x_{\sigma(1)}\otimes x_{\sigma(2)} \ldots x_{\sigma(n)}\] The signs in the summation are given by the degrees of $x_i$, for example in $S^2V$ is generated by elements of the form $x\otimes y + (-1)^{|x||y|} y\otimes x $.  The coproduct $\Delta$ restricts to a coproduct on $S^cV$. $S^cV$ is linearly isomorphic to $SV =V \oplus V\wedge V\oplus V\wedge V\wedge V\ldots $ by the map which sends $x_1\wedge x_2\wedge\ldots x_n$ in $SV$ to $\sum_{\sigma \in S_n} \pm x_{\sigma(1)} \otimes x_{\sigma(2)} \ldots x_{\sigma(n)}$ in $S^cV$. Note the product that $T^cV$ acquires from it's isomorphism to $TV$ does not restrict to the product $S^cV$ acquires from it's isomorphism to $SV$. Here onwards we will use $SV$ to denote both, the algebra with the usual product structure and the coalgebra with the induced coproduct structure.

\begin{lemma}
The map $\tilde{f}$ in (\ref{univ}) is given by the formula, \[ \tilde{f}(x) = f(x) + f\otimes f \circ \Delta_C (x) + f^{\otimes 3}\circ \Delta_C^2(x) + \ldots\] where $\Delta_C$ is the reduced coproduct on $C$. Since $C$ is nilpotent this sum is finite.
\label{lift}
\end{lemma}

\begin{proof}
Note that $\pi \circ \tilde{f}(x) = f(x) $. To prove the lemma we need to show that \[\Delta_T \circ \tilde{f} = \tilde{f}\otimes\tilde{f} \circ \Delta_C \]

The Sweedler notation for higher powers of $\Delta_C$ applied to $x$ is \[ \Delta_C^n (x) = \sum_{(x)} x_{(1)}\otimes x_{(2)}\otimes\ldots\otimes x_{(n+1)} \]

So the formula for $\tilde{f}$ as defined becomes \[\tilde{f}(x) = \sum_{n} \sum_{(x)} f(x_{(1)})\otimes f(x_{(2)})\otimes\ldots\otimes f(x_{(n+1)}) \] and when we apply the coproduct to this we get \[\Delta_T \circ \tilde{f}(x) = \sum_{n} \sum_{(x)} \sum_i f(x_{(1)})\otimes\ldots\otimes f(x_{(i)})\bigotimes\ldots\otimes f(x_{(n+1)}) \]

On the other hand we have 
\[\tilde{f}\otimes\tilde{f}\circ\Delta_C(x) = \sum_{(x)}\tilde{f}(x_{(1)})\otimes\tilde{f}(x_{(2)}) \]
\[ = \sum_{(x)}\sum_{i,j}f(x_{(1)})\otimes\ldots\otimes f(x_{(i)})\bigotimes\ldots\otimes f(x_{(i+j)}) \]
\[ = \sum_{n} \sum_{(x)} \sum_i f(x_{(1)})\otimes\ldots\otimes f(x_{(i)})\bigotimes\ldots\otimes f(x_{(n+1)}) \]
which is what is required.
\end{proof}

If $C$ is graded cocommutative then the image of $\tilde{f}$ is in $S^cV$ which we have identified with $SV$. So the formula in Lemma \ref{lift} becomes \[ \tilde{f}(x) = f(x) + f\wedge f \circ \Delta_C (x) + f^{\wedge 3}\circ \Delta_C^2(x) + \ldots\]

Suppose $A$ is a graded commutative associative algebra. Then consider the map $\tau:SA\rightarrow A$ which takes $x_1\wedge x_2\wedge\ldots x_n$ to $x_1x_2\ldots x_n$. We call this map the tautologous map. Since this is a map from a graded cocommutative coalgebra we can lift it to a map $\tilde{\tau}:SA\rightarrow SA$ and from Lemma \ref{lift} the formula for $\tilde{\tau}$ is given by the equation \ref{formula}.

\begin{lemma}
\label{iso}
$\tilde{\tau}:SA\rightarrow SA$ is a coalgebra isomorphism.
\end{lemma}

\begin{proof}
Let $F_n = \bigoplus_{i=1}^n \wedge^iA$ where $\wedge^iA$ is the graded symmetric product of $i$ copies of $A$. Then $F_n$ defines an increasing filtration of $SA$ and $SA$ is the direct limit of this sequence of sub-coalgebras. \[A = F_1 \hookrightarrow F_2 \hookrightarrow \ldots F_n \hookrightarrow \ldots \] For $v \in \wedge^nA$, $\tilde{\tau}$ is given by the formula \[ \tilde{\tau}(v) = \tau(v) + \tau \wedge \tau (\Delta(v)) + \ldots + \tau^{\wedge n} (\Delta^{n-1}(v)) \] Note that the highest degree term in the formula is $v$ and so this formula preserves the filtration $F_n$. We will show by induction that $\tilde{\tau}$ restricted to $F_n$ is an isomorphism for each $n$. Now on $F_1$, $\tilde{\tau}$ is identity hence an isomorphism. Now consider the following short exact sequence \[ 0 \rightarrow F_{n-1} \hookrightarrow F_n \rightarrow \wedge^nA \rightarrow 0\] On $F_{n-1}$, $\tilde{\tau}$ is an isomorphism by induction hypothesis and on $\wedge^nA$ the map induced is the identity map. So $\tilde{\tau}$ is an isomorphism on $F_n$. Since it's restriction to each $F_n$ is an isomorphism, $\tilde{\tau}$ is an isomorphism from $SA$ to $SA$. 
\end{proof}

A linear map from a graded commutative algebras $A$ to itself can be extended to either a coderivation or a coalgebra map from $SA$ to $SA$. These can then be conjugated by $\tilde{\tau}$ as it is an isomorphism.  The linear map itself doesn't have to be an algebra map or a derivation. 

\begin{Theorem}
\begin{itemize}
\item[i)]
Suppose $d$ is any linear map from a graded commutative algebra $A$ to $A$. Then there exists a unique coderivation $\tilde{d}$ from $SA$ to $SA$ such that the following diagram commutes. 

\begin{equation}
\label{coder}
\begin{tikzcd}
SA \arrow{r}{\tau} \arrow[swap]{d}{\tilde{d}} & A \arrow{d}{d} \\
SA \arrow{r}{\tau} & A
\end{tikzcd}
\end{equation}

The coderivation construction, $d$ gives $\tilde{d}$, preserves commutators. This implies that if $d$ has degree $-1$ and squares to zero then $\tilde{d}$ also has degree$-1$ and squares to zero.

\item[ii)]
Suppose $f$ is any linear map between graded commutative algebras $A$ to $B$. Then there exists a unique coalgebra morphism $\hat{f}$ from $SA$ to $SB$ such that the following diagram commutes. 

\begin{equation}
\label{hom}
\begin{tikzcd}
SA \arrow{r}{\tau_A} \arrow[swap]{d}{\hat{f}} & A \arrow{d}{f} \\
SB \arrow{r}{\tau_B} & B
\end{tikzcd}
\end{equation}

Note that the uniqueness implies that the coalgebra construction , $f$ gives $\hat{f}$, respects composition. 

\item[iii)]
(\cite{homotopy probability I} and \cite{homotopy probability II}) Suppose $A$ and $B$ are graded commutative algebras and $d_A$ and $d_B$ are linear maps, not necessarily derivations, of degree $-1$ and square zero on $A$ and $B$ respectively. Suppose $f$ is a linear map, not necessarily an algebra map, from $A$ to $B$ such that $f\circ d_A = d_B\circ f$. Then $\hat{f}$ as in ii) is the unique homomorphism of differential graded coalgebras from $(SA, \tilde{d_A})$ to $(SB,\tilde{d_B})$ such that the following diagram commutes.

\begin{equation}
\label{dgc:hom}
\begin{tikzcd}
(SA,\tilde{d_A}) \arrow{r}{\tau_A} \arrow[swap]{d}{\hat{f}} & (A,d_A) \arrow{d}{f} \\
(SB,\tilde{d_B}) \arrow{r}{\tau_B} & (B,d_B)
\end{tikzcd}
\end{equation}

\end{itemize}

\end{Theorem}

\begin{proof}
\begin{itemize}
\item[i)]
A linear map $d$ from $A$ to $A$ can be uniquely extended to a coderivation $D$ on $SA$ which commutes with the projection map to $A$ for the coalgebra structure on $SA$. That is there exists a unique $D$ such that the following diagram commutes. 

\begin{equation}
\begin{tikzcd}
SA \arrow{r}{\pi} \arrow[swap]{d}{\tilde{d}} & A \arrow{d}{d} \\
SA \arrow{r}{\pi} & A
\end{tikzcd}
\end{equation}

where $\pi$ is the projection map. $D$ is given by the following formula. \[D(x_1\wedge x_2\wedge\ldots x_n) = \sum_i \pm x_1\wedge\ldots\wedge d(x_i)\wedge\ldots x_n \]The signs in the sum depend on the degrees of $x_i$ and the degree of the map $d$. Consider the map $\tilde{d}:= \tilde{\tau}\circ D\circ\tilde{\tau}^{-1}$. This map is the unique coderivation on $SA$ which makes (\ref{coder}) commute.

If $d_1$ and $d_2$ are two linear maps of some degrees from $A$ to $A$ and let $D_1$ and $D_2$ be the coderivations on $SA$ which commute with the projection maps. Since the bracket of two coderivations is a coderivation, $[D_1, D_2]$ is a coderivation. Besides this is the unique coderivation which extends $[d_1, d_2]$. So, $[\tilde{d_1},\tilde{d_2}]$ is the unique coderivation with the above property for the map $[d_1,d_2]$. Thus if $d$ has degree $-1$ and squares to zero, then $\tilde{d}^2 = 1/2 [\tilde{d},\tilde{d}]$ is a coderivation which uniquely extends $d^2$. Since $d^2$ is zero, $\tilde{d}^2$ will also be zero. So maps of degree $-1$ which square to zero give coderivations which square to zero on $SA$. 

\item[ii)]
A linear map $f$ between $A$ and $B$ can be uniquely extended to a homomorphism of coalgebras $F$ from $SA$ to $SB$ which commutes with the projection maps. That is there exists a homomorphism of coalgebras such that the following diagram commutes. 

\begin{equation}
\begin{tikzcd}
SA \arrow{r}{\pi_A} \arrow[swap]{d}{F} & A \arrow{d}{f} \\
SB \arrow{r}{\pi_B} & B
\end{tikzcd}
\end{equation}

$F$ is given by the following formula. \[ F(x_1\wedge x_2\wedge\ldots x_n) = f(x_1)\wedge f(x_2)\wedge\ldots f(x_n)\] The map $\hat{f} = \tilde{\tau}\circ F\circ\tilde{\tau}^{-1}$ is then the unique coalgebra homomorphism which makes (\ref{hom}) commute.

\item[iii)] 
We only need to check that $\hat{f}\circ \tilde{d_A} = \tilde{d_B}\circ \hat{f}$. Let $D_A$ and $D_B$ be the coderivations on $SA$ and $SB$ respectively which commute with the projection maps and $F$ be the coalgebra homomorphism from $SA$ to $SB$ which commutes with the projections. Since $f\circ d_A = d_B\circ f$, by direct computation we have that $F\circ D_A = D_B\circ F$. It follows then that $\hat{f}\circ \tilde{d_A} = \tilde{d_B}\circ \hat{f}$. 

\end{itemize}
\end{proof}

Conjugation by $\tilde{\tau}$ of the lift to a coalgebra map measures the deviation of the original map from being an algebra homeomorphism and that of the lift to a coderivation measures how far the original map was from being a derivation for the algebra structure. 

\begin{Prop} Let $d$ be a linear map from $A$ to $A$ and $\tilde{d}$ be the coderivation constructed in the previous theorem. Let $f$ be a linear map from $A$ to $B$ and $\hat{f}$ be homomorphisms constructed from $f$ in the last theorem.   
\begin{itemize}
\item[i)]$d$ is a derivation if and only if \[\tilde{d}(x_1\wedge x_2\wedge\ldots x_n) = \sum_i \pm x_1x_2\ldots d(x_i)\ldots x_n \]
\item[ii)] $f$ is a homomorphism of algebras if and only if \[\hat{f}(x_1\wedge x_2\wedge\ldots x_n) = \tau(f(x_1)\wedge f(x_2)\wedge\ldots f(x_n)) = f(x_1)f(x_2)\ldots f(x_n)\]
\end{itemize}

\end{Prop}

\begin{proof}
\begin{itemize}
\item[i)]
In the proof of the last theorem we defined the coderivation $D$ such that it satisfies the formula \[\tau\circ D(x_1\wedge x_2\wedge\ldots x_n) = \sum_i \pm x_1x_2\ldots d(x_i)\ldots x_n \] This expression is equal to $d(x_1x_2\ldots x_n) = d\circ\tau(x_1\wedge x_2\wedge\ldots x_n)$ if and only if $d$ is a derivation for the algebra $A$. Thus we have that $d\circ\tau = \tau\circ D$ if and only if $d$ is a derivation for the algebra $A$. Since $\tilde{d}$ is the unique coderivation that commutes with $\tau$, $\tilde{d} = D$ if and only if $d$ is a derivation. 
\item[ii)]
Similar to the proof of the first part \[\tau\circ F(x_1\wedge x_2\wedge\ldots x_n) = \tau(f(x_1)\wedge f(x_2)\wedge\ldots f(x_n)) = f(x_1)f(x_2)\ldots f(x_n) \] which is equal to $f(x_1x_2\ldots x_n)$ if and only if $f$ is a homomorphism. By uniqueness of $\tilde{f}$ we have that it is equal to $F$ if and only if $f$ is a homomorphism. 

\end{itemize}
\end{proof} 

\section{More detailed Proof of Theorem \ref{app}}
\label{Proof}
\begin{proof}

The proof is similar to the proof of Lemma \ref{iso}. Define $\tilde{\tau}_C$ to be $\tilde{I}\circ\tilde{\tau}\circ\iota$. This map is a composition of differential coalgebra maps and hence is itself a differential coalgebra map. Let $F_n = \bigoplus_{i=1}^n\wedge^iC$, where $\wedge^iC$ is the graded symmetric product of $i$ copies of $C$. The ${F_n}$ is the filtration on $SC$. We will use induction on the degree of the filtration to prove that $\tilde{\tau}_C$ is an isomorphism. As $\iota$ preserves the filtration $\tilde{\tau}_C$ also preserves the filtration. Also as $\iota$ agrees with $i$ on $C=F_1$ and since $\tilde{\tau}$ is identity on $A$, $\tilde{\tau}_C$ is identity on $F_1$. Now suppose $\tilde{\tau}_C$ is an isomorphism when restricted to $F_{n-1}$. Since $\tilde{\tau}_C$ is a coalgebra mapping, for $v=x_1\wedge x_2\wedge\ldots x_n$ in $\wedge^nC$ we have that $\tilde{\tau}_C^{\wedge n}\circ\Delta^{n-1}(v)=\Delta^{n-1}\circ\tilde{\tau}_C(v) = v$. This implies that $\tilde{\tau}_C(v)=v+$ lower order terms. Then consider the short exact sequence \[ 0 \rightarrow F_{n-1} \hookrightarrow F_n \rightarrow \wedge^nC \rightarrow 0\] By induction hypothesis $\tilde{\tau}_C$ is an isomorphism on $F_{n-1}$ and by the above argument it induces identity on $\wedge^nC$. Hence it is an isomorphism on $F_n$ for every $n$ which implies it is a coalgebra isomorphism of $SC$. 

\end{proof}

\section{Related Work}

Our focus on the cumulant bijection was directly inspired by a lecture of Jae Suk Park at CUNY November 2011. The ideas in the lecture have been developed in the two papers \cite{homotopy probability I} and \cite{homotopy probability II}.

The ``induced cumulant bijection" is used in \cite{3D fluids} to set up potential algorithms for computing 3D fluid motion based on differential forms and the integration deformation retract to cochains.

\end{document}